\documentclass[11pt]{article}

\usepackage{amssymb,amsmath,amsthm,amsfonts,amscd}
\newtheorem{theorem}{Theorem}[section]

\newtheorem{proposition}{Proposition}[section]

\newtheorem{remark}[theorem]{Remark}

\begin{document}

\title{An iterative method to solve Lyapunov equations}

\author{Licio Hernanes Bezerra and Felipe Wisniewski (*)}
\date{Departamento de Matem\'atica, \\
Universidade Federal de S. Catarina,  Florian\'opolis, SC\\ Brazil 88040-900
licio.bezerra@ufsc.br \\
(*) Universidade Estadual do Paran\'a, Uni\~ao da Vit\'oria, PR \\ Brazil 84600-185
felipewisniewski@yahoo.com.br}

\maketitle

\begin{abstract}
We present here a new splitting method to solve Lyapunov equations of the type $AP + PA^T=-BB^T$ in a Kronecker product form. Although that resulting matrix is of order $n^2$, each
iteration of the method demands only two operations with the matrix $A$: a multiplication of the form $(A-\sigma I) \hat{B}$ and a inversion of the form $(A-\sigma I)^{-1}\hat{B}$.
We see that for some choice of a parameter the iteration matrix is such that all their eigenvalues are in absolute value less than 1, which means that it should converge without depending
on the starting vector. Nevertheless we present a theorem that enables us how to get a good starting vector for the method.
\end{abstract}

\maketitle

\section{Introduction}

Consider the following Lyapunov equation
\begin{equation}\label{lyap}
AP+PA^T=-BB^T,
\end{equation}
where $A \in \mathbb{R}^{n \times n}$ and $B \in \mathbb{R}^{n \times p}$.
Let $\lambda(A)$ denote the spectrum of $A$. If all the eigenvalues of $A$ have negative real part, we call $A$ a stable matrix. In this case,
there is an unique solution $P$ of equation \eqref{lyap} and $P$ is a symmetric semi-positive defined matrix.
The equation \eqref{lyap} appears in several engineering problems as can be seen in \cite{Sim2016}. In Mathematics, for instance,
it appears in the calculation of Gram matrices $P$ e $Q$ associated to a dynamical system given by the following equations:
$$
\Sigma\equiv\left\{\begin{array}{l}
\dot x(t)=Ax(t) + Bu(t),  \\
y(t)=Cx(t)+Du(t), \quad t \geq t_0, x(t_0)=x_0,
\end{array}\right.
$$
where $A \in \mathbb{R}^{n\times n},$ $B \in \mathbb{R}^{n\times m}$, $C \in \mathbb{R}^{p\times n},$ $D \in \mathbb{R}^{p\times m},$ $x(t) \in \mathbb{R}^n$, and $u(t) \in \mathbb{R}^m$,
which are well analysed in \cite{Ant}.

In \cite{Sim2016} the author presents an overview of the current state methods of resolution of Lyapunov equation.
For large-scale problems we can point out methods based on Krilov subspaces, which consist of projecting the problem onto a significantly smaller Krylov space and then solving the
reduced order matrix equation. In \cite{Sim2007} the author proposes to project onto a block-wise Krylov subspace of the type
\begin{equation}\label{Krylov_racional_introducao}
\mathbf{K}_l(A,B)=span\{B, A^{-1}B, AB, A^{-2}B, A^{2}B,  A^{-3}B,..., A^{l-2}B, A^{-(l-1)}B\}.
\end{equation}

Another method of resolution of equation \eqref{lyap} that is widely used is the ADI (Alternating Direction Implicit) method.
It appeared in 1968 \cite{Smith} and in 1988 a technique that consists in using a control parameter was added, which motivated the name of the method \cite{Wachspress}. Since then some variations of the method have arisen, for instance, \cite{Penzl}, \cite{FRM}. In \cite{Sim2011} it has been proved that under certain conditions ADI is equivalent to a Krylov method with the rational basis given in \eqref{Krylov_racional_introducao}.

Here we introduce a new stationary iterative methods for solving the Lyapunov equation formulated as a vector form by means of the Kronecker product. In \cite{Sim2016}, it is said  that approaches based on the Kronecker formulations were abandoned since that linear system is of order $n^2$, which would be very expensive computationally speaking.
However, each iteration of our method demands only two algebraic operations: $(A-\sigma I)^{-1}\hat B$ and $A$ $(A+\sigma I)\tilde B$, where $\hat B$ and $\tilde B$ are matrices of the same order than B, as discussed in the next section. Like Jacobi, Gauss-Seidel and other splitting methods, our method has limitations as regards convergence. Anyway, it is an available simple method to solve Lyapunov equations.

\section{The method}

We want to calculate a low rank matrix $P$ which is an approximate solution of (\ref{lyap}).
Note that the equation (\ref{lyap}) can be written as the following standard linear equation  $\tilde{A}{p}={b}$, where $\tilde{A}=(I \otimes A +A \otimes I),$ ${p}=\text{vec}(P)$ and ${b}=\text{vec}(-BB^T)$. The symbol $\otimes$ denotes here the Kronecker product of matrices and
$\text(vec)(X)$ is the usual representation of a matrix $X$ as a column vector.

Let $\sigma$ be a positive real number. Note that the equation (\ref{lyap}) is also equivalent to $\tilde{A}_{\sigma}{p}={b}$,
where $$\tilde{A}_{\sigma}=\left[I\otimes (A-\sigma I)+(A+\sigma I) \otimes I\right].$$

From this we can define
 the splitting as follows:
 $$\tilde{A}_{\sigma}=M_{\sigma}-N_{\sigma},$$ where $M_{\sigma}=I\otimes (A-\sigma I)$ and $N_{\sigma}=-(A+\sigma I) \otimes I$.

Note that if $A$ is a stable matrix and $\sigma$ is positive, then $(A-\sigma I)$ is inversible, and so is the matriz  $M_{\sigma}$. Moreover,
$$M_{\sigma}^{-1}=I \otimes (A-\sigma I)^{-1}.$$

From a starting vector ${p}_0 \in \mathbb{R}^{n^2}$,
we define for $k=0,1,2,...$ the following iterative method:

\begin{equation}\label{splittin_lyap_geral}
{p}_{k+1}=M_{\sigma}^{-1}N_{\sigma}{p}_{k}+M_{\sigma}^{-1}{b},
\end{equation}
where
$$
M_{\sigma}^{-1}N_{\sigma}=-(A+\sigma I)\otimes  (A-\sigma I)^{-1}.
$$
A very well-known theorem about this class of iteration \cite{GOL} tells us that, for all starting vector $p_0$, the method converges
if and only if the spectrum of $M_{\sigma}^{-1}N_{\sigma}$ is less than 1. In our case,

\begin{theorem}\label{conv_splitting_lyap}
The sequence defined by \eqref{splittin_lyap_geral} converges to the solution of equation \eqref{lyap} if and only if
\begin{equation}\label{razao_alpha}
\left\vert \frac{\lambda_i+\sigma}{\lambda_j-\sigma} \right\vert < 1, \quad \forall  \lambda_i,\lambda_j \in \lambda(A).
\end{equation}
\end{theorem}

Now, by analysing the iterations of the method, we can state the following result.

\begin{proposition}
Let $P_0= 0$. Then
\begin{eqnarray}\label{iteradas}
P_{k+1}&=&\sum_{i=1}^{k+1}(-1)^{i+1}(A-\sigma I)^{-i}BB^T\left((A+\sigma I)^T\right)^{(i-1)} \nonumber \\
&=& P_k + (-1)^{k+2}(A-\sigma I)^{-k-1}BB^T\left((A+\sigma I)^T\right)^{k}
\end{eqnarray}
\end{proposition}
\begin{proof}
$p_1 = M_{\sigma}^{-1}b = \left( I \otimes (A-\sigma I)^{-1}\right) b$, which is equivalent to $$P_1 = -(A-\sigma I)^{-1} BB^T.$$
Suppose that for $k\ge 1$
$$P_k = P_{k-1} + (-1)^{k+1}(A-\sigma I)^{-k}BB^T\left((A+\sigma I)^T\right)^{k-1}.$$
Since for all $k\ge 0$ ${p}_{k+1}=-(A+\sigma I)\otimes  (A-\sigma I)^{-1}{p}_{k}+ p_1$, that is,
$$
P_{k+1} = - (A-\sigma I)^{-1} P_k (A+\sigma I)^T + P_1,$$
we have that
$$P_{k+1} = - (A-\sigma I)^{-1} P_{k-1} (A+\sigma I)^T + $$
$$ + (-1)^{k+2} (A-\sigma I)^{-1} (A-\sigma I)^{-k}BB^T\left((A+\sigma I)^T\right)^{k-1} (A+\sigma I)^T + P_1=$$
$$ = - (A-\sigma I)^{-1} P_{k-1} (A+\sigma I)^T + (-1)^{k+2} (A-\sigma I)^{-k-1}BB^T\left((A+\sigma I)^T\right)^{k} + P_1.$$
Hence,
$$P_{k+1} = P_k + (-1)^{k+2} (A-\sigma I)^{-k-1}BB^T\left((A+\sigma I)^T\right)^{k},$$

\end{proof}

\begin{remark}
 Note that for each step it suffices to compute $(A-\sigma I)^{-1}[(A-\sigma I)^{-k}B]$ and $(A+\sigma I)[(A+\sigma I)^{k-1}] B$, where $[(A-\sigma I)^{-k}B]$ and $[(A+\sigma I)^{k-1}B]$ have already been calculated. Therefore, although the method has been developed to solve a system of order $n^2$, each iteration demands only two algebraic operations with a matrix of order $n$.
\end{remark}

Simple calculation yields the following result.

\begin{proposition}\label{sigma_splitting}
Let $\lambda(A) = \{\lambda_1,\lambda_2,...,\lambda_n\}$, that is, the spectrum of $A$. Suppose $A$ is stable, that is,
$Re \, (\lambda)<0$ for all $\lambda \in \lambda(A)$. Then the sequence defined by \eqref{splittin_lyap_geral} converges to the solution of equation \eqref{lyap} if and only if
for all $1\le i, j\le n$
$$ \sigma > \frac{|\lambda_j|^2 - |\lambda_i|^2}{2 \, (Re\, \lambda_j + Re\, \lambda_i)}.$$
\end{proposition}

\begin{remark} Since the matrix $A$ is real, without loss of generality we can suppose in the above inequality that $Im \, \lambda_i \ge 0$ and $Im \, \lambda_j \ge 0$.
Therefore, $$ \frac{|\lambda_j|^2 - |\lambda_i|^2}{2 \, (Re\, \lambda_j + Re\, \lambda_i)} =
\frac{|\lambda_j| + |\lambda_i|}{2 \, (Re\, \lambda_j + Re\, \lambda_i)} (|\lambda_j| - |\lambda_i|) \le$$
$$ \le \frac{|Re \, \lambda_j + Re\, \lambda_i| + (Im \, \lambda_j + Im \, \lambda_i)} {2 \, |Re\, \lambda_j + Re\, \lambda_i|} \left| |\lambda_j| - |\lambda_i|\right|. \label{sigma}$$
From another point of view,
$$ \frac{|\lambda_j|^2 - |\lambda_i|^2}{2 \, (Re\, \lambda_j + Re\, \lambda_i)} =
\frac{(Re \, \lambda_j)^2 - (Re \, \lambda_i)^2 + (Im \, \lambda_j)^2 - (Im \, \lambda_i)^2}{2 \, (Re\, \lambda_j + Re\, \lambda_i)} = $$
$$=\frac{Re \, \lambda_j - Re \, \lambda_i}{2} + \frac{Im \, \lambda_j + Im \, \lambda_i}{2 \, (Re\, \lambda_j + Re\, \lambda_i)} . (Im \, \lambda_j - Im \, \lambda_i).$$

In practical studies of stability, we just consider stable matrices that have their spectra contained in the left half-plane limited by lines of the type $y = \pm k x$.
For instance, it is usual to take $k\le 20$ in Brazilian electrical power systems stability analysis as safety margin (damping ratio of 5\%) \cite{GMP}.
In this case, from \ref{sigma}, it suffices that
$$\sigma \ge \min \left( \frac{21}{2} r(A),\max_{\lambda \in \lambda(A)} \frac{|Re \, \lambda|}{2} + 10 \max_{\lambda \in \lambda(A)} |Im \, \lambda|\right).$$
where $r(A)$ is the spectral radius of $A$, in order to guarantee convergence of the splitting method described above.
These parameters can be calculated, for instance, by methods that use M\"obius transforms \cite{lic}. MATLAB function $eigs$ also computes some subsets of eigenvalues,
e.g., $eigs(A,k,'largestreal')$ returns the k largest real part eigenvalues \cite{leh}.
\end{remark}

\section{Concluding remarks}

We introduced here an iteration method that starts with $p_0=0$. The  general formulae of the $kth$ iteration from another starting value is similar to the formulae (\ref{iteradas}) obtained with $p_0=0$. We could say that the formulae (\ref{lyap}) truely resulted from $p_1 = - M^{-1} b$.
Moreover, if $rank (B) = p$, each iteration realizes a rank-$p$ update.
We have computed approximate
solutions of different Lyapunov equations by using our method, typically applying to large sparse matrices that occur in the modelling of power system stability problems.
For starting vectors, we have followed the theorem below to calculate a good one.

\begin{theorem} \label{start}
Suppose that the columns of the $n\times p$ matrix $B$ in equation \eqref{lyap} are spanned by only $k$ eigenvectors of $A$, that is,
$B = V_k R_k$, where $V_k$ is a $n\times k$ matrix formed with those $k$ eigenvectors as columns, and $R_k$ is a $k\times p$ matrix.
Then the solution $P$ can be exactly computed.
\end{theorem}
\begin{proof}
Let $P=V_kX_kV_k^H$. Let $D_k$ the diagonal matrix such that $AV_k = V_kD_k$. Then $AP + PA^T=-BB^T$ can be written as
$$V_kD_kX_kV_k^H + V_kX_k \overline{D_k} V_k^H = - V_k R_k R_k^H V_k^H.$$
Hence,
$$D_kX_k + X_k \overline{D_k} = - R_k R_k^H.$$
Therefore, $X_k = - C_k \circ R_k R_k^H$, where $C_k$ is the $k\times k$ Cauchy matrix
$$C_k = \left( \begin{array}{cccc} \frac{1}{\lambda_1 + \overline{\lambda_1}} & \frac{1}{\lambda_1 + \lambda_1} & \cdots & \frac{1}{\lambda_1 + \overline{\lambda_k}} \\
                                   \frac{1}{\overline{\lambda_1} + \overline{\lambda_1}} & \frac{1}{\overline{\lambda_1}+\lambda_1} & \cdots & \frac{1}{\overline{\lambda_1} + \overline{\lambda_k}} \\
                                       \vdots & \vdots 7 \vdots & \vdots \\
                                     \frac{1}{\lambda_k + \overline{\lambda_1}} & \frac{1}{\lambda_k + \lambda_1} & \cdots & \frac{1}{\lambda_k + \overline{\lambda_k}}
                                     \end{array} \right),
$$
and $\circ$ denotes the Hadamard product of matrices.
\end{proof}

From Theorem \eqref{start}, if we first calculate an invariant subspace of $A$ close to the subspace spanned by the columns of $B$, it generally results a good starting vector for our method,
that is a vector that yields a fast convergence of the method after a few iterations. There are tests and comparisons between our methods and ADI, performances of the method with acceleration etc
in \cite{the}, where the method was first introduced from what we have known up to now.

\bibliographystyle{amsplain}

\end{document}